\documentclass[12pt,reqno]{amsart}
\usepackage{latexsym}
\usepackage{amssymb}
\usepackage{mathrsfs}
\usepackage{amsmath}
\usepackage{fancybox,color}
\usepackage{enumerate}
\usepackage{enumitem}
\usepackage[latin1]{inputenc}
\usepackage{soul} 
\usepackage[colorlinks=true, linkcolor=magenta, citecolor=magenta]{hyperref}

\usepackage{color}

\newcommand{\Mnc}{{M}^{*}} %noncentered
\newcommand{\ch}[1]{{\mbox{\hbox{$\mathbf{1}$}}}_{\hbox{$\scriptstyle #1$}}}

\usepackage{comment}

\numberwithin{equation}{section}

\makeatletter %modifies the appearance of subsection headings
\renewcommand\subsection{\@startsection{subsection}{2}%
  \z@{-.5\linespacing\@plus-.7\linespacing}{.3\linespacing}%
  {\normalfont\bfseries}}
\makeatother

\usepackage[left=2.1cm,top=2.3cm,right=2.1cm]{geometry}

\def\1{\raisebox{2pt}{\rm{$\chi$}}}

\newtheorem{theorem}[equation]{Theorem}
\newtheorem{corollary}[equation]{Corollary}
\newtheorem{lemma}[equation]{Lemma}

\theoremstyle{definition}
\newtheorem{definition}[equation]{Definition}

\newcommand{\R}{{\mathbb R}}

\newcommand{\N}{{\mathbb N}}

\newcommand{\eps}{{\varepsilon}}
\def\1{\raisebox{2pt}{\rm{$\chi$}}}

\def\vint_#1{\mathchoice%
        {\mathop{\kern 0.2em\vrule width 0.6em height 0.69678ex depth -0.58065ex
                \kern -0.8em \intop}\nolimits_{\kern -0.4em#1}}%
        {\mathop{\kern 0.1em\vrule width 0.5em height 0.69678ex depth -0.60387ex
                \kern -0.6em \intop}\nolimits_{#1}}%
        {\mathop{\kern 0.1em\vrule width 0.5em height 0.69678ex
            depth -0.60387ex
                \kern -0.6em \intop}\nolimits_{#1}}%
        {\mathop{\kern 0.1em\vrule width 0.5em height 0.69678ex depth -0.60387ex
                \kern -0.6em \intop}\nolimits_{#1}}}
\def\vintslides_#1{\mathchoice%
        {\mathop{\kern 0.1em\vrule width 0.5em height 0.697ex depth -0.581ex
                \kern -0.6em \intop}\nolimits_{\kern -0.4em#1}}%
        {\mathop{\kern 0.1em\vrule width 0.3em height 0.697ex depth -0.604ex
                \kern -0.4em \intop}\nolimits_{#1}}%
        {\mathop{\kern 0.1em\vrule width 0.3em height 0.697ex depth -0.604ex
                \kern -0.4em \intop}\nolimits_{#1}}%
        {\mathop{\kern 0.1em\vrule width 0.3em height 0.697ex depth -0.604ex
                \kern -0.4em \intop}\nolimits_{#1}}}

\newcommand{\aveint}[2]{\mathchoice%
        {\mathop{\kern 0.2em\vrule width 0.6em height 0.69678ex depth -0.58065ex
                \kern -0.8em \intop}\nolimits_{\kern -0.45em#1}^{#2}}%
        {\mathop{\kern 0.1em\vrule width 0.5em height 0.69678ex depth -0.60387ex
                \kern -0.6em \intop}\nolimits_{#1}^{#2}}%
        {\mathop{\kern 0.1em\vrule width 0.5em height 0.69678ex depth -0.60387ex
                \kern -0.6em \intop}\nolimits_{#1}^{#2}}%
        {\mathop{\kern 0.1em\vrule width 0.5em height 0.69678ex depth -0.60387ex
                \kern -0.6em \intop}\nolimits_{#1}^{#2}}}

\newcommand{\dist}{\operatorname{dist}}

\title[Self-improving properties of weighted norm inequalities]{Self-improving properties of weighted norm inequalities on metric measure spaces} 
\author[J.\! Kinnunen]{Juha Kinnunen}   %  can use \and
\address[J.K.]{Department of Mathematics, Aalto University, P.O. Box 11100, FI-00076 Aalto University, Finland}
\email{juha.k.kinnunen@aalto.fi}

\author[J.\! Lehrb\"ack]{Juha Lehrb\"ack}   %  can use \and
\address[J.L.]{University of Jyvaskyla, Department of Mathematics and Statistics, P.O. Box 35, FI-40014 University of Jyvaskyla, Finland}
\email{juha.lehrback@jyu.fi}

\author[A.V.\! V\"ah\"akangas]{Antti V. V\"ah\"akangas}
\address[A.V.V.]{University of Jyvaskyla, Department of Mathematics and Statistics, P.O. Box 35, FI-40014 University of Jyvaskyla, Finland} 
\email{antti.vahakangas@iki.fi}

\author[D.\! Yang]{Dachun Yang}   %  can use \and
\address[D.Y.]{Laboratory of Mathematics and Complex Systems (Ministry of Education of China), School of Mathematical Sciences, Beijing Normal University, Beijing 100875, The People's Republic of China}
\email{dcyang@bnu.edu.cn}

\date{\today}
\makeatletter
\@namedef{subjclassname@2020}{{\mdseries 2020} Mathematics Subject Classification}
\makeatother


\begin{document}

\begin{abstract}
This work discusses self-improving properties of the Muckenhoupt condition and weighted norm inequalities for 
the Hardy--Littlewood maximal function on metric measure spaces with a doubling measure.
Our main result provides direct proofs of these properties by applying a Whitney covering argument 
and a technique inspired by the Calder\'on--Zygmund decomposition.
In particular, this approach does not rely on reverse H\"older inequalities.
\end{abstract}

\maketitle

\section{Introduction}

Harmonic analysis on metric measure spaces, and more broadly on spaces of homogeneous type, continues to be an active area of research. Coifman and Weiss provided a systematic account of many aspects of harmonic analysis on spaces of homogeneous type in \cite{MR499948}. 
The theory of weighted norm inequalities and Muckenhoupt weights has been explored in greater detail in works such as those by Str\"omberg and Torchinsky in \cite{MR1011673} and by Genebashvili, Gogatishvili, Kokilashvili and Krbec in \cite{MR1791462}.
Muckenhoupt weights exhibit a remarkable self-improving property: 
if a weight belongs to a certain Muckenhoupt class $A_p$, $1<p<\infty$,
then it also belongs to another Muckenhoupt class $A_{p-\varepsilon}$ with a slightly smaller index, see \cite[Lemma~8, p.~5]{MR1011673}.
This self-improving nature of Muckenhoupt weights arises from the fact that these weights satisfy a reverse H\"older inequality, see \cite[Theorem~15, p.~9]{MR1011673} for details.

Reverse H\"older inequalities are powerful tools in the study of weighted norm inequalities, but 
the purpose of this work is to discuss an alternative approach to the self-improving property 
of the Muckenhoupt weights that entirely avoids relying on reverse H\"older inequalities.
Thus the novelty in this work lies in the methodology rather than in the actual results.

Instead of reverse H\"older inequalities, we utilize the close connection between Muckenhoupt weights and 
weighted norm inequalities for the Hardy--Littlewood maximal function, see \cite{MR2797562,MR807149}.
More specifically, let $(X,d,\mu)$ be a metric measure space equipped with a doubling measure $\mu$
and let $w$ be a weight in $X$; see Sections~\ref{s.prelim} and ~\ref{s.muck} for the definitions.
Then the following assertions are equivalent, for $1<p<\infty$.
\begin{enumerate}[label=\rm{(\Alph*)},topsep=6pt,itemsep=2pt]
\item\label{as1} There exists a constant $C>0$ such that
\begin{equation}\label{e.Bint1}
\int_{X} \bigl(\Mnc f(x)\bigr)^p w(x)\,d\mu(x) \le C\int_{X} \lvert f(x)\rvert^p w(x)\,d\mu(x)
\end{equation}
for every $f\in L^\infty(X)$ with a bounded support.

\item\label{as2} The weight $w$ belongs to the Muckenhoupt class $A_p$.

\item\label{as3} There exists a constant $C>0$ such that inequality \eqref{e.Bint1} holds for every $f\in L^p(w\,d\mu)$.
\end{enumerate}

In~\eqref{e.Bint1}, the
(noncentered) Hardy--Littlewood maximal function of $f$ is defined as
\[
 \Mnc f(x)= \sup_{B\ni x}\frac{1}{\mu(B)}\int_{B} \lvert f(y)\rvert \,d\mu(y),
\]
where the supremum is taken over all balls $B$ in $X$ containing $x$.
It was shown in \cite{MR499948} that $\Mnc$ is a bounded operator on 
 $L^p(X)$  for every $1<p\le\infty$ 
and of weak type $(1,1)$ under the assumption that $\mu$ is a doubling measure;
see also \cite[Theorem~3, p.~3]{MR1011673}. Observe that assertion \ref{as3} above states that
the operator $\Mnc$ is bounded on $L^p(w\,d\mu)$, for $1<p<\infty$.

The equivalence of the assertions \ref{as2} and \ref{as3} is well known,
see \cite[Theorem~9, p.~5]{MR1011673}. For our purposes, it is convenient to
consider also assertion \ref{as1}, which is a priori slightly weaker 
than~\ref{as3}. 
That \ref{as1} implies \ref{as2} is shown in Lemma \ref{l.Ap_easier},
and the more involved implication from \ref{as2} to \ref{as3} is proved in
Theorem~\ref{t.max_lerner} by adapting an argument of Lerner \cite{MR2399047}.
This approach is direct and, in particular, does not refer to reverse H\"older inequalities. 
Finally, the implication from \ref{as3} to \ref{as1} is a triviality.

Now, assume that assertion \ref{as3} holds, i.e., that
$\Mnc$ is bounded on some $L^p(w\,d\mu)$, $1<p<\infty$,
where $w$ is a weight and $\mu$ a doubling measure in $X$.
By the Marcinkiewicz interpolation theorem and the 
boundedness of the maximal operator $\Mnc$ on $L^\infty(w\,d\mu)$, 
it follows that
then $\Mnc$ is also bounded on $L^{p+\varepsilon}(w\,d\mu)$ for every $\varepsilon>0$.
On the other hand, by using the equivalence of the assertions \ref{as2} and \ref{as3}
and the known self-improving property of the Muckenhoupt $A_p$ condition, 
there exits some $\varepsilon>0$ such that
$\Mnc$ is bounded on $L^{p-\varepsilon}(w\,d\mu)$ as well.
%One way to prove this is to apply the fact that  in the sense that 
%if $w\in A_p$, then $w\in A_{p-\varepsilon}$ for some $\varepsilon>0$, see \cite[Lemma~8, p.~5]{MR1011673}
%Since \jl \eqref{e.Bint1}, for every $f\in L^p(w\,d\mu)$, is equivalent to \ed $w\in A_p$,  

The main goal in this work is to give a direct and essentially self-contained 
proof for this self-improving property of 
the weighted norm inequalities. In particular, we show in Lemma~\ref{l.max_impro}
that if assertion \ref{as3} holds, then
there exist $\varepsilon>0$ and $C>0$ such that
\begin{equation}\label{e.Bint2}
\int_{X} \bigl(\Mnc f(x)\bigr)^{p-\varepsilon}w(x)\,d\mu(x) \le C\int_{X} \lvert f(x)\rvert^{p-\varepsilon}w(x)\,d\mu(x)
\end{equation}
for every $f\in L^{\infty}(X)$ with a bounded support.
The proof of Lemma~\ref{l.max_impro} uses a Whitney covering argument 
and a technique inspired by the Calder\'on--Zygmund decomposition,
based on an unpublished note on Muckenhoupt weights by Keith and Zhong.
This approach extends the corresponding Euclidean result with the Lebesgue measure 
\cite[Theorem~8.11]{MR4306765} to metric spaces %with a doubling measure.
and it does not rely on reverse H\"older inequalities.
Similar ideas have been applied in various contexts, for example, in Keith and Zhong \cite{MR2415381}, Kinnunen, Lehrb\"ack, V\"ah\"akangas and Zhong \cite{MR3895752} and Lewis \cite{MR1239922}.

By the equivalence of the assertions \ref{as1}, \ref{as2} and \ref{as3}, 
we obtain from Lemma~\ref{l.max_impro} an independent proof for the
self-improving property of the Muckenhoupt classes, see Theorem~\ref{t.max_impro_ap},
and we also conclude that the improved weighted norm inequality \eqref{e.Bint2} holds for
all functions $f\in L^{p-\varepsilon}(w\,d\mu)$, see Corollary~\ref{c.other_way}.

%In particular, we obtain the following result for the weighted norm inequalities.
%\begin{theorem}\label{t.intro_wni}
%Let $(X,d,\mu)$ be a metric measure space equipped with a doubling measure $\mu$.
%Let $1<p<\infty$, let $w$ be a weight in $X$, and assume that
%there exists a constant $C>0$ such that 
%%\begin{equation}\label{e.Bint1}
%%\int_{X} \bigl(\Mnc f(x)\bigr)^p w(x)\,d\mu(x) \le C_1\int_{X} \lvert f(x)\rvert^p w(x)\,d\mu(x)
%%\end{equation}
%\eqref{e.Bint1} holds
%for every $f\in L^p(w\,d\mu)$.
%Then there exist $\varepsilon>0$ and a constant $C_2>0$ such that
%\[
%\int_{X} \bigl(\Mnc f(x)\bigr)^{p-\varepsilon}w(x)\,d\mu(x) \le C_2\int_{X} \lvert f(x)\rvert^{p-\varepsilon}w(x)\,d\mu(x)
%\]
%for every $f\in L^{p-\varepsilon}(w\,d\mu)$. [Like this or $L^\infty$-version?]
%\end{theorem}

Before the statements and proofs of our main results and tools in Sections~\ref{s.max} and~\ref{s.si_wne},
we recall preliminaries on metric spaces and on Muckenhoupt weights in Sections~\ref{s.prelim} and~\ref{s.muck},
respectively. In Section~\ref{s.muck} we also prove some basic properties of the Muckenhoupt classes for the
sake of completeness.

\subsection*{Acknowledgements} The authors would like to thank Xiao Zhong 
for interesting discussions on the topics of the paper. 
Dachun Yang is supported by the National Natural Science Foundation of China (Grant numbers 12431006 and 12371093).

%\subsection*{Acknowledgements}
%The research is supported by the Academy of Finland.

\section{Preliminaries}\label{s.prelim}

Throughout the paper, we assume that $X=(X,d,\mu)$ is a metric measure space equipped with a metric $d$ and a 
positive complete Borel
measure $\mu$ such that $0<\mu(B)<\infty$
for all open balls $B\subset X$, where  \[B=B(x,r)=\{y\in X\,:\, d(y,x)<r\}\] with $x\in X$ and $r>0$. 
%As in \cite[p.~2]{MR2867756},
%we extend $\mu$ as a Borel regular (outer) measure on $X$.
The space $X$ is separable
under the above assumptions, we refer to \cite[Proposition~1.6]{MR2867756}.
We  also assume throughout that $\# X\ge 2$ and  
that the measure $\mu$ is doubling, that is,
there exists a constant $c_\mu> 1$, called
the doubling constant of $\mu$, such that
\begin{equation}\label{e.doubling}
\mu(2B) \le c_\mu\, \mu(B)
\end{equation}
for all balls $B=B(x,r)$ in $X$. 
Here we use for $0<t<\infty$ the notation $tB=B(x,tr)$. 
We let
\[
\vint_{A} f(y)\,d\mu(y)=\frac{1}{\mu(A)}\int_A f(y)\,d\mu(y)
\]
be the integral average of $f\in L^1(A)$ over a measurable set $A\subset X$
with $0<\mu(A)<\infty$. We say that a function $f\colon X\to\R$ has a bounded
support, if $\{x\in X\,:\, f(x)\not=0\}$ is contained in some ball $B$ of $X$. 
The characteristic function of a set $E\subset X$
is denoted by $\ch{E}$; that is, $\ch{E}(x)=1$ if $x\in E$
and $\ch{E}(x)=0$ if $x\in X\setminus E$.

%\begin{definition}\label{def.non_centered_maximal}
% of a measurable real-valued function $f$ on $X$
% is defined by
%\[
%\Mnc f(x) = \sup_{B\ni x}\vint_B \lvert f(y)\rvert\,d\mu(y)\,,\qquad x\in X\,,
%\]
%where the supremum is taken over all balls $B\subset X$ with $x\in B$.
%\end{definition}

\section{Muckenhoupt weights}\label{s.muck}

For convenience, we recall the standard definition and elementary properties of Muckenhoupt weights in the setting of metric measure spaces. We follow \cite{MR807149} and also refer to \cite{MR4306765}.

\begin{definition}\label{def.semilocal}
A  measurable function $w\colon X\to\R$ is a weight in $X$, if  $w$ is integrable
on all balls of $X$  and 
$w(x)>0$ for almost
every $x\in X$.  If $w$ is a weight in $X$ and $E\subset X$ is a measurable set, we write 
\begin{equation}\label{e.w_meas}
w(E)=\int_E w(x)\,d\mu(x).
\end{equation}
We say that a weight $w$  in $X$ 
is doubling, if there exists a constant $c_w\ge 1$ such that
\[
0<w(B(x,2r))\le c_w w(B(x,r))<\infty
\]
for all $x\in X$ and $r>0$.
\end{definition}

\begin{definition}\label{d.A_p}
A weight $w$ in $X$ belongs to the Muckenhoupt class 
$A_p$, for $1<p<\infty$, if
there exists a constant $C$ such that
\begin{equation}\label{e.A_pd}
\vint_B w(x)\,d\mu(x) \biggl(\vint_B w(x)^{\frac{1}{1-p}}\,d\mu(x)\biggr)^{p-1} \le C
\end{equation}
for every ball $B\subset X$.
The smallest possible constant $C$ in \eqref{e.A_pd} is called the $A_p$ constant of $w$, and it is 
denoted by $[w]_{A_p}$. 
\end{definition}

 If $1<p<q<\infty$, then H\"older's inequality shows that $A_p \subset A_q$. 
Next we recall some other basic properties of Muckenhoupt weights
that will be applied later. We begin with a duality statement. 

%The following basic properties of Muckenhoupt weights follow rather directly from the definitions and H\"older's inequality.
%The first assertion is that the Muckenhoupt classes are nested and the second assertion is a duality statement. 

%\begin{lemma}\label{l.approp}
%The following assertions hold.
%\begin{enumerate}[label=\rm{(\alph*)}]%,leftmargin=18pt
%\item\label{it.approp1} For $1<p<q<\infty$, we have $A_p \subset A_q$.
%\item\label{it.approp2} For $1<p<\infty$, we have $w \in A_p$ if and only if $w^{1-p'} \in A_{p'}$, where $p'=\frac p{p-1}$.
%\end{enumerate}
%\end{lemma}

\begin{lemma}\label{l.approp}
Let $1<p<\infty$ and write $p'=\frac p{p-1}$. Then
$w \in A_p$ if and only if $w^{1-p'} \in A_{p'}$.
\end{lemma}

\begin{proof}
%To prove assertion \ref{it.approp1}, let $1<p<q<\infty$,
%let $B\subset X$ be a ball, and assume that $w \in A_p$. 
%By H\"older's inequality
%\begin{align*}
%\biggl(\vint_{B}w(x)^{\frac{1}{1-q}}\,d\mu(x)\biggr)^{q-1}
%&\leq\biggl(\vint_{B}\bigl(w(x)^{\frac{1}{1-q}}\bigr)^{\frac{q-1}{p-1}}\,d\mu(x)\biggr)^{(q-1)\frac{p-1}{q-1}}
%\\
%&=\biggl(\vint_{B}w(x)^{\frac{1}{1-p}}\,d\mu(x)\biggr)^{p-1}.
%\end{align*}
%This implies that
%\begin{align*}
%&\vint_B w(x)\,d\mu(x)\biggl(\vint_B w(x)^{\frac{1}{1-q}}\,d\mu(x)\biggr)^{q-1}
%\\
%&\qquad\le\vint_B w(x)\,d\mu(x)\biggl(\vint_B w(x)^{\frac{1}{1-p}}\,d\mu(x)\biggr)^{p-1} 
%\le[w]_{A_p},
%\end{align*}
%and assertion \ref{it.approp1} follows.
%To prove assertion \ref{it.approp2}, let $1<p<\infty$ and 
Assume that $w\in A_p$. 
Then
\begin{align*}
&\vint_{B}w(x)^{1-p'}\,d\mu(x)
\biggl(\vint_{B}\bigl(w(x)^{1-p'}\bigr)^{\frac{1}{1-p'}}\,d\mu(x)\biggr)^{p'-1}
\\&\qquad =\vint_{B}w(x)^{1-p'}\,d\mu(x)\biggl(\vint_{B}w(x)\,d\mu(x)\biggr)^{\frac{1}{p-1}}
\leq [w]_{A_p}^{\frac1{p-1}}
\end{align*}
for every ball $B\subset X$.
This shows that $w^{1-p'} \in A_{p'}$.
Conversely, if $w^{1-p'} \in A_{p'}$ then $w=(w^{1-p'})^{1-p}\in A_p$ by the same argument.
\end{proof}

Next we consider a weighted H\"older inequality for integral averages. 

\begin{lemma}\label{l.K_est}
Let $1<p<\infty$ and assume that $w\in A_p$.
Let $B\subset X$ be a ball and $f\in L^1(B)$ be a nonnegative function. 
%Let $1<p<\infty$ and assume that $w\in A_p$. 
%If $Q\subsetX$ is a cube and 
%$f\in L^\infty(Q)$ is a nonnegative function,
Then 
\[
\vint_B f(x)\,d\mu(x)
\le [w]_{A_p}^{\frac{1}{p}} \biggl(\frac{1}{w(B)}\int_B f(x)^p w(x)\,d\mu(x)\biggr)^{\frac{1}{p}}.
\]
\end{lemma}

\begin{proof}
H\"older's inequality and~\eqref{e.A_pd} imply
\begin{align*}
\vint_B f(x)\,d\mu(x) 
%=\frac{1}{\lvert Q\rvert}\int_Q f(x)\,dx \\
&\le \frac{1}{\mu(B)}\biggl(\int_B w(x)^{\frac{1}{1-p}}\,d\mu(x)\biggr)^{\frac{p-1}{p}}\biggl(\int_B f(x)^p w(x)\,d\mu(x)\biggr)^{\frac{1}{p}} \\
&=\biggl(\frac{w(B)}{\mu(B)}\biggr)^{\frac{1}{p}}\biggl(\vint_B w(x)^{\frac{1}{1-p}}\,d\mu(x)\biggr)^{\frac{p-1}{p}}
  \biggl(\frac{1}{w(B)}\int_B f(x)^p w(x)\,d\mu(x)\biggr)^{\frac{1}{p}}\\
&\le [w]_{A_p}^{\frac{1}{p}} \biggl(\frac{1}{w(B)}\int_B f(x)^p w(x)\,d\mu(x)\biggr)^{\frac{1}{p}}. \qedhere
\end{align*}
\end{proof}

Since the underlying measure $\mu$ is doubling, 
an application of Lemma~\ref{l.K_est} shows that
a Muckenhoupt weight in $X$ is doubling.

\begin{lemma}\label{l.Ap_dbl}
Let $1<p<\infty$ and assume that $w\in A_p$. Then 
$w$ is doubling with a constant 
 $c_w=C(c_\mu,p,[w]_{A_p})$.
\end{lemma}

\begin{proof}
Let $x\in X$ and $r>0$.
The doubling condition in \eqref{e.doubling} and  Lemma~\ref{l.K_est}
with $f=\ch{B(x,r)}$ give
\begin{align*}
\frac{1}{c_\mu}&\le \frac{\mu(B(x,r))}{\mu(B(x,2r))} 
=\vint_{B(x,2r)} f(y)\,d\mu(y)
\\&\le [w]_{A_p}^{\frac{1}{p}} \biggl(\frac{1}{w(B(x,2r))}\int_{B(x,2r)} f(y)^p w(y)\,d\mu(y)\biggr)^{\frac{1}{p}}\\
&\le [w]_{A_p}^{\frac{1}{p}}\left(\frac{w(B(x,r))}{w(B(x,2r))}\right)^{\frac{1}{p}}.
\end{align*}
The claim follows by raising both sides
to power $p$ and reorganizing the terms.
\end{proof}

\section{Maximal function estimates}\label{s.max}

Let $w$ be a weight in $X$. For $1\le p<\infty$,  
the weighted $L^p$-norm of a measurable function $f$ on $X$ is defined as
\[
\lVert f\rVert_{L^p(w\,d\mu)}=\left(\int_X \lvert f(x)\rvert^pw(x) \,d\mu(x)\right)^{\frac{1}{p}},
\]
and for $p=\infty$ we define
\[
\lVert f\rVert_{L^\infty(w\,d\mu)}
= \inf\bigl\{ c\ge 0 : \lvert f(x)\rvert \le c \ w\text{-a.e.\ } x\in X \bigr\}.
\]
If $\lVert f\rVert_{L^p(w\,d\mu)}<\infty$, we write
$f\in L^p(w\,d\mu)$, and in the case $w=1$ we use the notation
$L^p(d\mu) = L^p(X)$.

\begin{definition}\label{def.non_centered_maximal}
Let $w$ be a doubling weight in $X$.
For a measurable function $f$ on $X$, 
the weighted maximal function $M^{*,w} f$ is defined by
\[
M^{*,w} f(x)= \sup_{B\ni x}\frac{1}{w(B)}\int_{B} \lvert f(y)\rvert w(y)\,d\mu(y),
\]
where the supremum is taken over all balls $B$ in $X$ with $x\in B$.
When $w=1$, we obtain the noncentered maximal function $\Mnc f=M^{*,1}f$.
\end{definition}

The following simple lemma gives a sufficient condition for the Muckenhoupt
property. Observe that the
class of test functions is restricted to essentially bounded functions with a bounded support.
This restriction is convenient for later purposes.

\begin{lemma}\label{l.Ap_easier}
Let $1<p<\infty$ and assume that $w$ is a weight in $X$. Assume that there exists
a constant $C_1$ such that 
\begin{equation}\label{e.Bint}
\int_{X} \bigl(\Mnc f(x)\bigr)^p w(x)\,d\mu(x) \le C_1\int_{X} f(x)^p w(x)\,d\mu(x)
\end{equation}
for every nonnegative $f\in L^\infty(X)$ with a bounded support. Then $w\in A_p$ and $[w]_{A_p}\le C_1$.
\end{lemma}

\begin{proof}
Fix a ball $B\subset X$. Let $\eps>0$ and $f=(\eps+w)^{\frac{1}{1-p}}\ch{B}$. 
Observe that $f\in L^\infty(X)$ and $f$ has a bounded support.
Clearly
\begin{align*}
\vint_{B} (\eps+w(y))^{\frac{1}{1-p}}\,d\mu(y)=\vint_{B} \lvert f(y)\rvert\,d\mu(y)\le \Mnc f(x)
\end{align*}
for every $x\in B$, 
and thus 
\begin{align*}
\vint_{B} w(x)\,d\mu(x) \biggl(\vint_{B} (\eps+w(y))^{\frac{1}{1-p}}\,d\mu(y)\biggr)^p
&\le \vint_{B} \bigl(\Mnc f(x)\bigr)^p w(x)\,d\mu(x).
\end{align*}
By the assumption, we have 
\begin{align*}
\vint_{B} \bigl(\Mnc f(x)\bigr)^p w(x)\,d\mu(x) 
&\le \frac1{\mu(B)}\int_{X} \bigl(\Mnc f(x)\bigr)^p w(x)\,d\mu(x) \\
&\le C_1\frac1{\mu(B)}\int_{X} f(x)^p w(x)\,d\mu(x) \\
&  \le C_1\vint_{B} (\eps+w(x))^{\frac{p}{1-p}}(\eps+w(x))\,d\mu(x)   \\
&= C_1\vint_{B} (\eps+w(x))^{\frac{1}{1-p}}\,d\mu(x).
\end{align*}
This gives 
\begin{align*}
\vint_{B} w(x)\,d\mu(x) \biggl(\vint_{B} (\eps+w(y))^{\frac{1}{1-p}}\,d\mu(y)\biggr)^{p-1}
\le C_1
\end{align*}
for every $\eps>0$. Fatou's lemma implies
\begin{align*}
&\vint_{B} w(x)\,d\mu(x) \biggl(\vint_{B} w(x)^{\frac{1}{1-p}}\,d\mu(x)\biggr)^{p-1}
 \\&\qquad \le \liminf_{\eps\to 0+}\vint_{B} w(x)\,d\mu(x) \biggl(\vint_{B} (\eps+w(x))^{\frac{1}{1-p}}\,d\mu(x)\biggr)^{p-1}
\le C_1,
\end{align*}
and the claim follows.
\end{proof}

Next we aim to show a strengthened converse of Lemma \ref{l.Ap_easier},
see Theorem \ref{t.max_lerner} below.
Since a doubling weight $w$ in $X$ can be interpreted as a doubling measure as in 
Definition~\ref{def.semilocal},
we have the following weighted version of the Hardy--Littlewood maximal function theorem,
see for instance \cite[Theorem~3.13]{MR2867756}.

\begin{lemma}\label{l.M_doubling_bdd}
Let $1<p<\infty$ and let $w$ be a doubling weight in $X$, with a constant $c_w$.
Then there exists a constant $C=C(p,c_w)$ such that 
\[
\lVert M^{*,w} f\rVert_{L^p(w\,d\mu)}\le C\lVert f\rVert_{L^p(w\,d\mu)}
\] 
for every $f\in L^p(w\,d\mu)$. 
\end{lemma}

The proof of the weighted norm inequality in Theorem~\ref{t.max_lerner}
is adapted from \cite{MR2399047}. 
We also refer to \cite[Theorem~8.19]{MR807149}. 

\begin{theorem}\label{t.max_lerner}
Let $1<p<\infty$ and $w\in A_p$. There exists a constant $C=C(c_\mu,p,[w]_{A_p})$ such that
\[
\int_{X} \bigl(\Mnc f(x)\bigr)^p w(x)\,d\mu(x) \le C\int_{X} \lvert f(x)\rvert^p w(x)\,d\mu(x)
\]
for every $f\in L^p(w\,d\mu)$.
\end{theorem}

\begin{proof}
Define $\sigma=w^{1-p'}=w^{\frac{1}{1-p}}$.
Lemma~\ref{l.approp}  implies that $\sigma\in A_{p'}$, where $p'=\frac{p}{p-1}$.
Let $f\in L^p(w\,d\mu)$. Let $x\in X$ and consider a ball $B=B(x_0,r)$ such that $x\in B$. 
Then
\begin{align*}
\vint_B \lvert f(y)\rvert\,d\mu(y)=
A
\left(\frac{\mu(B)}{w(B)}\left(\frac{1}{\sigma(B)}\int_B\lvert f(y)\rvert\,d\mu(y)\right)^{p-1}\right)^{\frac{1}{p-1}},
\end{align*}
where
\begin{align*}
A=\left(\frac{w(B)\sigma(B)^{p-1}}{\mu(B)^p}\right)^{\frac{1}{p-1}}
%&\le \left(\frac{w(B)\sigma(B)^{p-1}}{\lvert Q(x_0,3r)\rvert^p}\right)^{\frac{1}{p-1}}
\le [w]_{A_p}^{\frac{1}{p-1}}.
\end{align*}
Let $z\in B$. Then
\begin{align*}
\frac{1}{\sigma(B)}\int_B\lvert f(y)\rvert\,d\mu(y)
&=\frac{1}{\sigma(B)}\int_{B}\lvert f(y) \sigma(y)^{-1}\rvert \sigma(y)\,d\mu(y)
\le M^{*,\sigma} (f\sigma^{-1})(z).
\end{align*}
Hence we have
\begin{align*}
\vint_B \lvert f(y)\rvert\,d\mu(y)
&\le [w]_{A_p}^{\frac{1}{p-1}}
\left(\frac{1}{w(B)}\int_B \bigl(M^{*,\sigma}(f\sigma^{-1})(z)\bigr)^{p-1} w(z)^{-1} w(z)\,d\mu(z)\right)^{\frac{1}{p-1}}\\
%&\le  C(n,p,[w]_{A_p})
%\left(\frac{1}{w(B(x,2r))}\int_{B(x,2r)} M^{*,\sigma}(f\sigma^{-1})(z)^{p-1} w(z)^{-1} w(z)\,dz\right)^{\frac{1}{p-1}}\\
&\le
[w]_{A_p}^{\frac{1}{p-1}} \bigl(M^{*,w}\bigl((M^{*,\sigma}(f\sigma^{-1}))^{p-1} w^{-1}\bigr)(x)\bigr)^{\frac{1}{p-1}}.
\end{align*}
By taking supremum over all balls $B$ as above, it follows that
\[
\Mnc f(x)\le [w]_{A_p}^{\frac{1}{p-1}} 
\bigl(M^{*,w}\bigl((M^{*,\sigma}(f\sigma^{-1}))^{p-1} w^{-1}\bigr)(x)\bigr)^{\frac{1}{p-1}}
\]
for all $x\in X$.  Observe that $w\in A_p$ and $\sigma\in A_{p'}$ are doubling weights by Lemma~\ref{l.Ap_dbl}, 
with constants $c_w$ and $c_\sigma$
depending only on $c_\mu$, $p$ and $[w]_{A_p}$.
Lemma \ref{l.M_doubling_bdd} implies
\begin{align*}
\int_{X} \bigl(\Mnc f(x)\bigr)^p w(x)\,d\mu(x)
&\le  [w]_{A_p}^{\frac{p}{p-1}} 
\int_{X}\bigl(M^{*,w}\bigl((M^{*,\sigma}(f\sigma^{-1}))^{p-1} w^{-1}\bigr)(x)\bigr)^{\frac{p}{p-1}} w(x)\,d\mu(x)\\
&\le C(c_\mu,p,[w]_{A_p})\int_{X}\bigl(M^{*,\sigma}(f\sigma^{-1})(x)\bigr)^{p} \sigma(x)\,d\mu(x)\\
&\le C(c_\mu,p,[w]_{A_p})\int_{X}\lvert f(x)\rvert^p w(x)\,d\mu(x),
\end{align*}
where the facts $\sigma=w^\frac{1}{1-p}$ and $w=\sigma^{1-p}$ are used.
\end{proof}

\section{Self-improvement results}\label{s.si_wne}

We need the following Whitney type covering theorem, see \cite[Proposition~4.1.15]{MR3363168} and \cite{MR499948}.

\begin{lemma}\label{l.whitney_dec}
Assume that $\Omega\not=\emptyset$ is an open subset of $X$ such that $X\setminus \Omega\not=\emptyset$.
There exist balls $B_i=B(x_i,r_i)\subset\Omega$, $i\in\N$, with
\begin{equation}\label{e.dist_est}
r_i=\frac{1}{8}\dist(x_i,X\setminus\Omega),
\end{equation}
such that $\Omega=\bigcup_{i=1}^\infty B_i$ and 
\begin{equation}\label{e.cover_bound}
\sum_{i=1}^\infty \ch{B_i}(x)\le C(c_\mu)
\end{equation}
for every $x\in\Omega$. 
%where
%\begin{equation}\label{e.dist_est}
%r_i=\frac{1}{8}\dist(x_i,X\setminus\Omega).
%\end{equation}
\end{lemma}

%We will need a {\em Whitney ball cover} 
%$\mathcal{W}_0=\mathcal{W}(B_0)$ of  $B_0\subsetneq X$.
%This countable family with good covering properties
%is comprised of the so-called {\em Whitney balls} that are of the form $Q=B(x_Q,r_Q)\in\mathcal{W}_0$,
%with center $x_Q\in B_0$ and radius
%\[
%r_Q=\frac{\dist(x_Q, X\setminus B_0)}{128}>0\,.
%\]
%The $4$-dilated Whitney ball is denoted by $Q^*=4Q=B(x_Q,4 r_Q)$ whenever $Q\in\mathcal{W}_0$.
%Even though the Whitney balls need not be pairwise disjoint, they nevertheless have 
%the following standard covering properties with bounded overlap; cf.\ \cite[pp.~77--78]{MR2867756}:
%\begin{itemize}
%\item[(W1)] $B_0=\bigcup_{Q\in\mathcal{W}_0} Q$;
%\item[(W2)] $\sum_{Q\in\mathcal{W}_0} \mathbf{1}_{Q^*}\le C\mathbf{1}_{B_0}$ for some constant $C=C(c_\nu)>0$.
%\end{itemize}

%A collection $\{Q_i:i\in\N\}$ of cubes $Q_i\subset\Omega$ 
%satisfying the properties in Lemma~\ref{l.whitney_dec}
%is called a Whitney decomposition\index{Whitney decomposition} of $\Omega$.
%We remark that the claims concerning the cubes $Q_i^*$ in
%Lemma~\ref{l.whitney_dec} are in fact consequences of the properties
%of the cubes $Q_i$, see Stein~\cite[Chapter~6]{Stein1970} for details.

Let $f\in L^\infty(X)$ be a nonnegative function with
a bounded support. Let $t>0$ be such that
\[E_t=\{x\in X : \Mnc f(x)>t\}\not=X,\] where $\Mnc f$ is
the noncentered maximal function of $f$ as in Definition~\ref{def.non_centered_maximal}.
Then $E_t\subsetneq X$ is an open set. We define
a function $f_t$ as below.
\begin{itemize}
\item
If $E_t\not=\emptyset$, consider  a Whitney type cover $\{B_i : i\in\N\}$
of $E_t$ as in Lemma~\ref{l.whitney_dec}, and let 
%Then the half-open dyadic cubes $Q_i$, $i\in\N$, are pairwise disjoint, they cover the open set $E_t$, and 
%\begin{equation}\label{e.dist_est_ap}
%\diam(Q_i)\le \dist(Q_i,\partial E_t)\le 4 \diam(Q_i)
%\end{equation}
%for every $i\in\N$. 
\begin{equation}\label{d.ul}
f_t(x)=\begin{cases}
f(x),\qquad &x\in X\setminus E_t,\\
\displaystyle\sup_{B_i\ni x}\vint_{B_i} f(y)\,d\mu(y),\qquad &x\in E_t,
\end{cases}
\end{equation}
where the supremum is taken over all balls $B_i$, $i\in\N$, such that $x\in B_i$.
\item If $E_t=\emptyset$, let $f_t(x)=f(x)$ for every $x\in X$. 
\end{itemize}

\begin{lemma}\label{l.curious_dom}
Assume that $f\in L^\infty(X)$ is a nonnegative function with
a bounded support. Let $t>0$ and let $E_t\subsetneq X$ and $f_t$ be as defined above.
Then the following assertions hold:
\begin{enumerate}[label=\rm{(\alph*)}] %,leftmargin=18pt
\item\label{it.lambda_ineq} 
$f_t(x)\le C(c_\mu)t$ for almost every $x\in X$.
%\begin{equation}\label{e.star2}
\item\label{it.star2} $\Mnc f(x)\le C(c_\mu)\Mnc f_t(x)$
%\end{equation}
for every $x\in X\setminus E_t$.
\end{enumerate}
\end{lemma}

\begin{proof}
By the Lebesgue differentiation theorem \cite[Section 3.4]{MR3363168}, we have
\[
f_t(x)=f(x)\le \Mnc f(x)\le t
\]
for almost every $x\in X\setminus E_t$. If $E_t=\emptyset$, then \ref{it.lambda_ineq} holds.
If $E_t\not=\emptyset$ and $x\in E_t$, then consider any ball $B_i$, $i\in\N$, such that $x\in B_i$. Recall that $E_t$ is the union of such balls. By~\eqref{e.dist_est} and \eqref{e.doubling}, there exists a ball $B\subset X$ such that
$B\setminus E_t\not=\emptyset$,
$B_i\subset B$ and $\mu(B)\le c_\mu^4\mu(B_i)$. 
For any $z\in B\setminus E_t$ we have
\[
\vint_{B_i} f(y)\,d\mu(y) 
\le c_\mu^4\vint_{B} f(y)\,d\mu(y)
\le c_\mu^4\Mnc f(z)\le c_\mu^4t.
\]
By taking supremum over all  balls $B_i$ as above, we see that $f_t(x)\le c_\mu^4t$ for every $x\in E_t$, and thus \ref{it.lambda_ineq} holds also in the case $E_t\not=\emptyset$.

To prove part~\ref{it.star2}, we may clearly assume that $E_t\not=\emptyset$.
Let $x\in X\setminus E_t$.
Let $z\in X$ and $r>0$ be such that $x\in B(z,r)$.  Then 
\begin{align*}
\int_{B(z,r)} f(y)\,d\mu(y)&=\int_{B(z,r)\setminus E_t} f(y)\,d\mu(y) + \int_{B(z,r)\cap E_t} f(y)\,d\mu(y)\\
&\le\int_{B(z,r)\setminus E_t} f(y)\,d\mu(y) + \sum_{i=1}^\infty \int_{B(z,r)\cap B_i} f(y)\,d\mu(y).
\end{align*}
Using the fact that $0\le f=f_t$ in $X\setminus E_t$, we get
\[
\int_{B(z,r)\setminus E_t} f(y)\,d\mu(y)\le \int_{B(z,2r)\setminus E_t} f(y)\,d\mu(y)=\int_{B(z,2r)\setminus E_t} f_t(y)\,d\mu(y).
\]
Next we consider any ball $B_i$ 
%If $Q_i\subset B(x,r)$, then
%\[
%\int_{B(y,r)\cap Q_i} f(z)\,dz = \int_{Q_i} f(z)\,dz=\int_{Q_i} f_t(y)\,d\mu(y)\,.
%\]
satisfying $B(z,r)\cap B_i\not=\emptyset$, if such $B_i$ exists. Since $x\in B(z,r)\setminus E_t$, 
by definition \eqref{e.dist_est} of the radius $r_i$, we have $B_i\subset B(z,2r)$. Thus~\eqref{d.ul} implies
\begin{align*}
\int_{B(z,r)\cap B_i} f(y)\,d\mu(y) & \le \int_{B_i} f(y)\,d\mu(y)=\int_{B_i} \vint_{B_i} f(y)\,d\mu(y)\, d\mu(z)\\
&\le \int_{B_i} f_t(z)\,d\mu(z)=\int_{B(z,2r)\cap B_i} f_t(y)\,d\mu(y).
\end{align*}
By collecting the estimates above and using \eqref{e.cover_bound}, we obtain
\begin{align*}
\int_{B(z,r)} f(y)\,d\mu(y)&\le \int_{B(z,2r)\setminus E_t} f_t(y)\,d\mu(y) + \sum_{i=1}^\infty \int_{B(z,2r)\cap B_i} f_t(y)\,d\mu(y) 
%\\& 
\\&\le C(c_\mu)\int_{B(z,2r)} f_t(y)\,d\mu(y),
\end{align*}
and consequently
\[
\vint_{B(z,r)} f(y)\,d\mu(y)\le C(c_\mu)\vint_{B(z,2r)} f_t(y)\,d\mu(y)
\le C(c_\mu)\Mnc f_t(x).
\]
By taking supremum over all balls $B(z,r)$ as above, we
conclude that~\ref{it.star2} holds. 
\end{proof}

Next we state and prove our main results concerning the
self-improving properties of
weighted norm inequalities for the
non-centered maximal function. 
To emphasize the key aspects of the arguments, 
we first formulate this property
only in terms of essentially bounded functions with a bounded support. 
We will later strengthen this result and recover
it for all $f\in L^{p-\varepsilon}(w\,d\mu)$, see Corollary~\ref{c.other_way}.

\begin{lemma}\label{l.max_impro}
Let $1<p<\infty$ and let $w$ be a weight in $X$.
Assume that there exists a constant $C_1$ such that 
\begin{equation}\label{e.assumed_M}
\int_{X} \bigl(\Mnc f(x)\bigr)^p  w(x)\,d\mu(x) \le  C_1 \int_{X}\lvert f(x)\rvert ^p w(x)\,d\mu(x)
\end{equation}
for all $f\in L^p(w\,d\mu)$. 
Then there exist constants $\eps=\eps(c_\mu,p,C_1)$ and $C=C(c_\mu,p,C_1)$ 
such that $0<\eps<p-1$ and
\[
\int_{X} \bigl(\Mnc f(x)\bigr)^{p-\eps} w(x)\,d\mu(x) \le C\int_{X} f(x)^{p-\eps} w(x)\,d\mu(x)
\]
for all nonnegative $f\in L^\infty(X)$ with a bounded  support. 
\end{lemma}

\begin{proof}
Fix a nonnegative function $f\in L^\infty(X)$ with a  bounded   support.
Let $t>0$ be such that  $E_t=\{x\in X : \Mnc f(x)>t\}\subsetneq X$ and let $f_t$ be as in~\eqref{d.ul}. 
By using  assumption \eqref{e.assumed_M}, we get
\[
\int_{X} \bigl(\Mnc f_t (x)\bigr)^p  w(x)\,d\mu(x)
\le  C_1  \int_{X} f_t(x)^p w(x)\,d\mu(x).
\]
Employing also Lemma~\ref{l.curious_dom}, 
we obtain
\begin{equation}\label{e.star} 
\begin{split}
\int_{X\setminus E_t} \bigl(\Mnc f(x)\bigr)^p w(x)\,d\mu(x) & \le C(c_\mu,p)\int_{X\setminus E_t} \bigl(\Mnc f_t(x)\bigr)^p  w(x)\,d\mu(x) \\
&\le C(c_\mu,p)\int_{X} \bigl(\Mnc f_t(x)\bigr)^p w(x)\,d\mu(x)\\
&\le C(c_\mu,p,C_1)\biggl(\int_{X\setminus E_t} f(x)^p w(x)\,d\mu(x)+t^p w(E_t)\biggr).
\end{split}
\end{equation}

We consider parameters $t_0>0$ and $0<\eps<p-1$. We will choose the latter parameter more specifically at the end of the proof. We have
\begin{equation}\label{e.collash}
\begin{split}
&\int_{t_0}^\infty t^{-1-\eps}\int_{X\setminus E_t} \bigl(\Mnc f(x)\bigr)^p  w(x)\,d\mu(x)\,dt\\
&\qquad\le C(c_\mu,p,C_1)\biggl(\int_{t_0}^\infty t^{-1-\eps}\int_{X\setminus E_t} f(x)^p w(x)\,d\mu(x)\,dt
+\int_{t_0}^\infty  t^{p-1-\eps}w(E_t)\,dt\biggr). 
%&\qquad\qquad+C(n,p,C_1)\int_{t_0}^\infty  t^{p-1-\eps}w(E_t)\,dt.
\end{split}
\end{equation}
By Fubini's theorem,
\begin{align*}
&\int_{t_0}^\infty t^{-1-\eps}\int_{X\setminus E_t} \bigl(\Mnc f(x)\bigr)^p  w(x)\,d\mu(x)\,dt \\
&\qquad =\int_{X} \bigl(\Mnc f(x)\bigr)^p\biggl(\int_{t_0}^\infty  \ch{X\setminus E_t}(x) t^{-1-\eps} \, dt\biggr)\,  w(x)\,d\mu(x)\\
& \qquad =\int_{X} \bigl(\Mnc f(x)\bigr)^p\biggl(\int_{\max\{t_0,\Mnc f(x)\}}^\infty  t^{-1-\eps} \, dt\biggr)\,  w(x)\,d\mu(x)\\
& \qquad =\frac{1}{\eps}\int_{X} \bigl(\Mnc f(x)\bigr)^p \bigl(\max\bigl\{t_0,\Mnc f(x)\bigr\}\bigr)^{-\eps} w(x)\,d\mu(x).
\end{align*}
In a similar way, we obtain  
\begin{align*}
\int_{t_0}^\infty t^{-1-\eps}\int_{X\setminus E_t} f(x)^p  w(x)\,d\mu(x)\,dt 
& =\frac{1}{\eps}\int_{X} f(x)^p \bigl(\max\bigl\{t_0,\Mnc f(x)\bigr\}\bigr)^{-\eps}  w(x)\,d\mu(x)\\
& \le \frac{1}{\eps}\int_{X} f(x)^{p-\eps}  w(x)\,d\mu(x),
\end{align*}
where the inequality holds since $0\le f(x)\le \max\bigl\{t_0,\Mnc f(x)\bigr\}$ for almost every $x\in X$.

The last term in~\eqref{e.collash} can be estimated as 
\begin{align*}
\int_{t_0}^\infty  t^{p-1-\eps}w(E_t)\,dt 
& =\int_{X} \biggl(\int_{t_0}^\infty \ch{E_t}(x)t^{p-1-\eps}\,dt\biggr) \,w(x)\,d\mu(x)\\
&=\int_{E_{t_0}} \biggl(\int_{t_0}^{\Mnc f(x)} t^{p-1-\eps}\,dt\biggr) \,w(x)\,d\mu(x)\\
&\le \frac{1}{p-\eps}\int_{E_{t_0}} \bigl(\Mnc f(x)\bigr)^{p-\eps}  w(x)\,d\mu(x)\\
&\le\frac{1}{p-\eps}\int_{X} \bigl(\Mnc f(x)\bigr)^{p} \bigl(\max\bigl\{t_0,\Mnc f(x)\bigr\}\bigr)^{-\eps} w(x)\,d\mu(x).
\end{align*}
Multiplying of the obtained inequalities by $\eps>0$
and using~\eqref{e.collash} gives 
\begin{equation}\label{e.collas}
\begin{split}
&\int_{X}  \bigl(\Mnc f(x)\bigr)^p  \bigl(\max\bigl\{t_0,\Mnc f(x)\bigr\}\bigr)^{-\eps}  w(x)\,d\mu(x)
\le C(c_\mu,p,C_1)\int_{X} f(x)^{p-\eps} w(x)\,d\mu(x)\\ 
&\qquad +C(c_\mu,p,C_1)\frac{\eps}{p-\eps}\int_{X} \bigl(\Mnc f(x)\bigr)^{p} \bigl(\max\bigl\{t_0,\Mnc f(x)\bigr\}\bigr)^{-\eps}w(x)\,d\mu(x).
\end{split}
\end{equation}
We fix $0<\eps<p-1$, only depending on $c_\mu$, $p$ and $C_1$, such that 
\[
C(c_\mu,p,C_1)\frac{\eps}{p-\eps}\le \frac{1}{2}.
\] 
Recall that $f\in L^\infty(X)$ has a bounded support
and $w$ is integrable on balls.
Thus the last term on the right-hand side of \eqref{e.collas} is finite 
by the assumption in~\eqref{e.assumed_M} and the fact that $t_0>0$.
Hence, we may absorb this term
to the left-hand side of \eqref{e.collas}, and we obtain the estimate
\begin{align*}
&\int_{X}  \bigl(\Mnc f(x)\bigr)^p  \bigl(\max\bigl\{t_0,\Mnc f(x)\bigr\}\bigr)^{-\eps}  w(x)\,d\mu(x)\le C(c_\mu,p,C_1)\int_{X} f(x)^{p-\eps} w(x)\,d\mu(x)
\end{align*}
that holds uniformly for all $t_0>0$.
Finally, by Fatou's lemma,
\begin{align*}
\int_{X}  \bigl(\Mnc f(x)\bigr)^{p-\eps} w(x)\,d\mu(x) 
& \le  \liminf_{t_0\to 0+}\int_{X}  \bigl(\Mnc f(x)\bigr)^p  \bigl(\max\bigl\{t_0,\Mnc f(x)\bigr\}\bigr)^{-\eps} w(x)\,d\mu(x)\\
& \le C(c_\mu,p,C_1)\int_{X} f(x)^{p-\eps} w(x)\,d\mu(x),
\end{align*}
and this concludes the proof.
\end{proof}

We can now prove a self-improving property of
Muckenhoupt weights.

\begin{theorem}\label{t.max_impro_ap}
Let $1<p<\infty$ and assume that $w\in A_p$. 
Then there exist $\eps=\eps(c_\mu,p,[w]_{A_p})$  such that $0<\eps<p-1$ 
and $w\in A_{p-\eps}$ and $[w]_{A_{p-\eps}}\le C(c_\mu,p,[w]_{A_p})$. Moreover,
\begin{equation}\label{e.final}
\int_{X} \bigl(\Mnc f(x)\bigr)^{p-\eps} w(x)\,d\mu(x) \le C(c_\mu,p,[w]_{A_p})\int_{X} \lvert f(x)\rvert^{p-\eps} w(x)\,d\mu(x)
\end{equation}
for all $f\in L^{p-\varepsilon}(w\,d\mu)$.
\end{theorem}

\begin{proof}
Theorem~\ref{t.max_lerner} and Lemma~\ref{l.max_impro} combined
show that there are constants $\eps=\eps(c_\mu,p,[w]_{A_p})$ and $C_1=C(c_\mu,p,[w]_{A_p})$ 
such that $0<\eps<p-1$ and
\[
\int_{X} \bigl(\Mnc f(x)\bigr)^{p-\eps} w(x)\,d\mu(x) \le C_1\int_{X} f(x)^{p-\eps} w(x)\,d\mu(x)
\]
for all nonnegative $f\in L^\infty(X)$ with a bounded  support. 
Lemma \ref{l.Ap_easier}
implies that $w\in A_{p-\eps}$ and
$[w]_{A_{p-\eps}}\le C_1=C(c_\mu,p,[w]_{A_p})$.
Finally, the inequality in \eqref{e.final} holds for all $f\in L^{p-\varepsilon}(w\,d\mu)$
by Theorem~\ref{t.max_lerner}.
\end{proof}
%Let $B$ be a ball in $X$. By Theorem
%\ref{t.max_impro_ap}
%there exist constants $\eps=\eps(c_\mu,p,[w]_{A_p})$ and $C_1=C(c_\mu,p,[w]_{A_p})$ such that $0<\eps<p-1$ and
%\begin{align*}
%\int_{B} \bigl(\Mnc f(x)\bigr)^{p-\eps} w(x)\,d\mu(x)
%&\le \int_{X} \bigl(\Mnc f(x)\bigr)^{p-\eps} w(x)\,d\mu(x)\le C_1\int_{X}  f(x)^{p-\eps} w(x)\,d\mu(x)
%\end{align*}
%for all nonnegative  $f\in L^\infty(X)$ with a bounded  support. Lemma \ref{l.Ap_easier}
%implies that $w\in A_{p-\eps}$ and
%$[w]_{A_{p-\eps}}\le C_1=C(c_\mu,p,[w]_{A_p})$.
%\end{proof}

By combining the results in Lemma~\ref{l.Ap_easier} and Theorem~\ref{t.max_impro_ap},
we obtain the following stronger variant of Lemma~\ref{l.max_impro},
in which the improved version of the weighted norm inequality
holds for all functions $f\in L^{p-\varepsilon}(w\,d\mu)$.

\begin{corollary}\label{c.other_way}
Let $1<p<\infty$ and let $w$ be a weight in $X$.
Assume that there exists a constant $C_1$ such that 
\[
\int_{X} \bigl(\Mnc f(x)\bigr)^p  w(x)\,d\mu(x) \le  C_1 \int_{X}f(x)^p w(x)\,d\mu(x)
\]
for all nonnegative functions $f\in L^\infty(X)$ with a bounded support.
Then there exist constants $\eps=\eps(c_\mu,p,C_1)$ and $C=C(c_\mu,p,C_1)$ 
such that $0<\eps<p-1$ and
\[
\int_{X} \bigl(\Mnc f(x)\bigr)^{p-\eps} w(x)\,d\mu(x) \le C\int_{X} \lvert f(x)\rvert^{p-\eps} w(x)\,d\mu(x)
\]
for all $f\in L^{p-\varepsilon}(w\,d\mu)$. 
\end{corollary}

\def\cprime{$'$} \def\cprime{$'$} \def\cprime{$'$}

\end{document}